\documentclass{article}%
\usepackage{amsmath}
\usepackage{amsfonts}
\usepackage{amssymb}
\usepackage{graphicx}
\usepackage[colorlinks=true, pdfstartview=FitV, linkcolor=blue,citecolor=red,
urlcolor=blue,a4paper]{hyperref}
\usepackage{color}%
\providecommand{\U}[1]{\protect\rule{.1in}{.1in}}
\newtheorem{theorem}{Theorem}

\newtheorem{corollary}[theorem]{Corollary}

\newtheorem{remark}[theorem]{Remark}

\newenvironment{proof}[1][Proof]{\noindent\textbf{#1.} }{\ \rule{0.5em}{0.5em}}
\allowdisplaybreaks
\begin{document}

\title{A New Expression for the Bernoulli Numbers and Its Applications}
\author{Levent Karg\i n and Merve Mutluer\thanks{Corresponding author}\\{\small Department of Mathematics, Akdeniz University, 07058-Antalya,
T\"{u}rkiye}\\{\small lkargin@akdeniz.edu.tr and mervemutluer34@gmail.com}}
\date{}
\maketitle

\begin{abstract}
This paper shows that a finite discrete convolution involving Stirling numbers
of both kinds and harmonic numbers can be expressed in terms of the Bernoulli
numbers. As applications of this expression, the linear recurrence relation
for the Bernoulli numbers given by Agoh is reproved, and a new recurrence
relation for the Bernoulli numbers is obtained. Furthermore, it is shown that
a cumulative sum of the Bernoulli numbers can be written in terms of the
Bernoulli and di-Bernoulli numbers. Finally, congruences for the sums of the
Bernoulli and Euler numbers are established.



\end{abstract}

\section{Introduction}

\setcounter{theorem}{0} \setcounter{equation}{0}

The Bernoulli numbers, denoted by $B_{n}$, first arose in the context of
computing sums of powers of positive integers in closed form, a topic that has
been studied since antiquity. They may be defined by means of the generating
function
\[
\frac{t}{e^{t}-1}=\sum_{n=0}^{\infty}B_{n}\frac{t^{n}}{n!},\text{\ \ \ \ }%
|t|<2\pi.
\]
It is easy to find the values $B_{0}=1$, $B_{1}=-1/2$, $B_{2}=1/6$, $B_{3}=0$,
$B_{4}=-1/30$, and $B_{2n+1}=0$ for all $n\geq1$. They also appear in the
Taylor series expansions of the tangent and hyperbolic tangent functions, the
Euler--Maclaurin summation formula, and certain representations of the Riemann
zeta function. The Bernoulli numbers have consequently become one of the
cornerstones of mathematical science, with numerous approaches available for
their introduction, including generating functions, recurrence relations,
integral representations, and explicit formulas. Among the numerous known
explicit formulas, one of the most famous is Worpitzky's identity%

\begin{equation}
\sum_{k=1}^{n}\left(  -1\right)  ^{k}{n\brace k} \frac{k!}{k+1}=B_{n},
\label{5}%
\end{equation}
which was later generalized to
\begin{equation}
\sum_{k=j}^{n}\frac{\left(  -1\right)  ^{k-j}}{k}{n\brace k} {k\brack j}
=\binom{n-1}{j}\frac{B_{n-j}}{n-j}, \label{4}%
\end{equation}
where $1\leq j\leq n$ (cf. \cite[p.217]{QG}). Here $%
\genfrac{[}{]}{0pt}{0}{n}{k}%
$ and $%
\genfrac{\{}{\}}{0pt}{0}{n}{k}%
$ denote the Stirling numbers of the first and second kind, defined by
\begin{equation}
\left(  x\right)  ^{\bar{n}}=x\left(  x+1\right)  \cdots\left(  x+n-1\right)
=\sum_{k=0}^{n}{n\brack k} x^{k} \label{1tss}%
\end{equation}
and
\[
x^{n}=\sum_{k=0}^{n}\left(  -1\right)  ^{n-k}{n\brace k} \left(  x\right)
^{\bar{k}},
\]
respectively (see, for example \cite[Chapters 9 and 12]{QG}), and $\dbinom
{n}{j}$ is the binomial coefficient.

In 2010, Mez\H{o} \cite{M2010} extended the identity (\ref{4}) and derived a
formula for the Bernoulli polynomials by using the $r$-Whitney numbers.
Motivated by this approach, Merca further showed that a finite discrete
convolution involving different generalizations of the Stirling numbers of
both kinds can be expressed in terms of the Bernoulli polynomials
\cite{Merca2014, Merca2015b}. In a recent paper \cite{AK}, new explicit
expressions for the Bernoulli numbers, namely,%
\begin{align}
\sum_{k=1}^{n}\left(  -1\right)  ^{k-1}%
\genfrac{\{}{\}}{0pt}{0}{n}{k}%
\left(  k-1\right)  !H_{k}  &  =B_{n-1},\label{6}\\
\sum_{k=2}^{n}\left(  -1\right)  ^{k}%
\genfrac{\{}{\}}{0pt}{0}{n}{k}%
\left(  k-1\right)  !H_{k-1}H_{k}  &  =\frac{n+1}{2}B_{n-2}, \label{7}%
\end{align}
were introduced. Here
\[
H_{n}=1+\frac{1}{2}+\frac{1}{3}+\cdots+\frac{1}{n}
\]
is the $n$th harmonic number. The identities
\begin{equation}
\left(  k-1\right)  !=%
\genfrac{[}{]}{0pt}{0}{k}{1}%
\text{ and }\left(  k-1\right)  !H_{k-1}=%
\genfrac{[}{]}{0pt}{0}{k}{2}
\label{20}%
\end{equation}
motivate us to examine the closed form of the sum
\[
\sum_{k=j}^{n}\left(  -1\right)  ^{k-j}%
\genfrac{\{}{\}}{0pt}{0}{n}{k}%
\genfrac{[}{]}{0pt}{0}{k}{j}%
H_{k}.
\]
As stated in the following theorem, this sum gives rise to a new formula for
the Bernoulli numbers.

\begin{theorem}
[Main Theorem]\label{TEO}For integers $n$ and $j$ with $0\leq j\leq n$, we
have%
\begin{equation}
\sum_{k=j}^{n}\left(  -1\right)  ^{k-j}%
\genfrac{\{}{\}}{0pt}{0}{n}{k}%
\genfrac{[}{]}{0pt}{0}{k}{j}%
H_{k}=\left(  \dbinom{n}{j}-1\right)  \frac{B_{n-j}}{n-j}. \label{1}%
\end{equation}

\end{theorem}

We note that for $j=1$ and $j=2$, the formula (\ref{1}) reduces to (\ref{6})
and (\ref{7}), respectively. Moreover, using the following known relation
\[%
\genfrac{[}{]}{0pt}{0}{k}{3}%
=\frac{1}{2!}\left(  k-1\right)  !\left[  \left(  H_{k-1}\right)  ^{2}%
-H_{k-1}^{\left(  2\right)  }\right]  ,
\]
where
\[
H_{n}^{\left(  m\right)  }=1+\frac{1}{2^{m}}+\cdots+\frac{1}{n^{m}}
\]
is the generalized harmonic number, we reach at the following special case of
(\ref{1}):%
\[
\sum_{k=3}^{n}\left(  -1\right)  ^{k-1}%
\genfrac{\{}{\}}{0pt}{0}{n}{k}%
\left(  k-1\right)  !\left[  \left(  H_{k-1}\right)  ^{2}-H_{k-1}^{\left(
2\right)  }\right]  H_{k}=\frac{n^{2}+2}{3}B_{n-3}.
\]
As applications of the main theorem, we reprove the linear recurrence relation
for the Bernoulli numbers given by Agoh \cite{AGOH2016} and derive a new
recurrence relation for the Bernoulli numbers. Moreover, we introduce a closed
formula for cumulative sum of the Bernoulli numbers. This formula leads us to
deduce some congruences for the sums of Bernoulli and Euler numbers modulo an
odd prime and its square. These are presented in the next section. The proof
of Theorem \ref{TEO} is given in the final section.

\section{Applications of the Main Theorem}

\setcounter{theorem}{0} \setcounter{equation}{0}

Multiplying both sides of (\ref{1}) with $x^{j}$ and summing over $j$ from $1$
to $n$ and using (\ref{1tss}) we obtain
\begin{equation}
\sum_{j=1}^{n}\left(  \dbinom{n}{j}-1\right)  \frac{B_{j}}{j}x^{n-j}=\text{
}_{H}w_{n}^{\left(  x\right)  }-H_{n}x^{n}, \label{9}%
\end{equation}
where $n$ is a positive integer, $x$ is a non-zero complex variable and%
\begin{equation}
_{H}w_{n}^{\left(  x\right)  }=\sum_{k=1}^{n}%
\genfrac{\{}{\}}{0pt}{0}{n}{k}%
\dbinom{x}{k}k!H_{k}. \label{17}%
\end{equation}

In this section, we list some special cases of (\ref{9}). First we have the
following linear identity which was also proved by a different method in
\cite{AGOH2016}.

\begin{corollary}
\label{cor1}For any positive integers $n$ and $m$, we have
\begin{equation}
\sum_{j=1}^{n}\left(  \dbinom{n}{j}-1\right)  \frac{B_{j}}{j}m^{n-j}%
=m^{n}\left(  H_{m}-H_{n}\right)  -\sum_{j=1}^{m}\frac{\left(  m-j\right)
^{n}}{j}. \label{12}%
\end{equation}

\end{corollary}

\begin{proof}
$x=m\geq1$ in (\ref{9}) gives
\begin{equation}
\sum_{j=1}^{n}\left(  \dbinom{n}{j}-1\right)  \frac{B_{j}}{j}m^{n-j}=\text{
}_{H}w_{n}^{\left(  m\right)  }-H_{n}m^{n}. \label{10}%
\end{equation}
To examine $_{H}w_{n}^{\left(  m\right)  }$, we first use the generating
function of the Stirling numbers of the second kind%
\begin{equation}
\sum_{n=k}^{n}%
\genfrac{\{}{\}}{0pt}{0}{n}{k}%
\frac{t^{n}}{n!}=\frac{\left(  e^{t}-1\right)  ^{k}}{k!}, \label{2sgf}%
\end{equation}
and obtain%
\[
\sum_{n=1}^{\infty}\text{ }_{H}w_{n}^{\left(  m\right)  }\frac{t^{n}}{n!}%
=\sum_{k=1}^{m}\dbinom{m}{k}H_{k}\left(  e^{t}-1\right)  ^{k}.
\]
The equality
\begin{equation}
\sum_{k=1}^{m}\dbinom{m}{k}H_{k}\left(  z-1\right)  ^{k}=H_{m}z^{m}-\sum
_{k=0}^{m-1}\frac{z^{k}}{m-k} \label{11}%
\end{equation}
(\cite[Eq. (2.4, iii)]{Agoh2021}) gives%
\[
\sum_{n=1}^{\infty}\text{ }_{H}w_{n}^{\left(  m\right)  }\frac{t^{n}}%
{n!}=H_{m}\sum_{n=0}^{\infty}m^{n}\frac{t^{n}}{n!}-\sum_{n=0}^{\infty}\left(
\sum_{k=0}^{m-1}\frac{k^{n}}{m-k}\right)  \frac{t^{n}}{n!}.
\]
Comparing the coefficients of $t^{n}/n!$ yields to%
\[
_{H}w_{n}^{\left(  m\right)  }=H_{m}m^{n}-\sum_{j=1}^{m}\frac{\left(
m-j\right)  ^{n}}{j}.
\]
Inserting this equation in (\ref{10}) completes the proof of (\ref{12}).
\end{proof}

Note that (\ref{12}) can also be written as
\[
\sum_{j=1}^{n}\left(  -1\right)  ^{j}\left(  \dbinom{n}{j}-1\right)
\frac{B_{j}}{j}m^{n-j}=m^{n}\left(  H_{m}-H_{n}+\frac{\left(  n-1\right)  }%
{m}\right)  -\sum_{j=1}^{m}\frac{\left(  m-j\right)  ^{n}}{j},
\]
since $\left(  -1\right)  ^{n}B_{n}=B_{n}$ for $n\geq2$. In particular, we
have
\[
\sum_{j=1}^{n}\left(  -1\right)  ^{j}\left(  \dbinom{n}{j}-1\right)
\frac{B_{j}}{j}=n-H_{n}.
\]
Moreover, If we set $m=1$ and $m=2$ in (\ref{12}) and combine the resulting
formulas, we arrive at%
\[
\sum_{j=1}^{n}\left(  \dbinom{n}{j}-1\right)  \frac{B_{j}}{j}\left(
1-2^{-j}\right)  =\frac{1-2^{n-1}}{2^{n}}.
\]
A slightly different form of the above equation can serve as a recurrence
relation for the Bernoulli numbers.

\begin{corollary}
For a positive integer $n$, we have
\begin{equation}
\sum_{j=1}^{n}\left(  \dbinom{n}{j}+1\right)  \frac{B_{j}}{j}\left(
1-2^{j}\right)  =1. \label{16}%
\end{equation}

\end{corollary}

\begin{proof}
We first set $x=-1/2$ in (\ref{9}) to obtain
\[
\sum_{j=1}^{n}\left(  -1\right)  ^{j}\left(  \dbinom{n}{j}-1\right)
\frac{B_{j}}{j}2^{j}=\left(  -1\right)  ^{n}2^{n}\text{ }_{H}w_{n}^{\left(
-1/2\right)  }-H_{n}.
\]
Since
\[
\sum_{k=1}^{\infty}\binom{2k}{k}H_{k}t^{k}=\frac{2}{\sqrt{1-4t}}\ln\left(
\frac{1+\sqrt{1-4t}}{2\sqrt{1-4t}}\right)  ,
\]
(\cite[Theorem 1, Eq. (4)]{B}) and
\[
\binom{-1/2}{k}=\frac{\left(  -1\right)  ^{k}}{2^{2k}}\dbinom{2k}{k},
\]
we obtain that
\[
\sum_{n=0}^{\infty}\text{ }_{H}w_{n}^{\left(  -1/2\right)  }\frac{\left(
-2t\right)  ^{n}}{n!}=2e^{t}\ln\left(  \frac{e^{t}+1}{2}\right)  .
\]
From the generating function of Euler polynomials \cite[Section 15.3]{QG}
\[
\sum_{n=0}^{\infty}E_{n}\left(  x\right)  \frac{t^{n}}{n!}=\frac{2}{e^{t}%
+1}e^{xt},
\]
we deduce that%
\[
\sum_{n=1}^{\infty}E_{n-1}\left(  1\right)  \frac{t^{n}}{n!}=2\ln\left(
\frac{e^{t}+1}{2}\right)  .
\]
Hence, we reach at
\[
\left(  -1\right)  ^{n}2^{n}\text{ }_{H}w_{n}^{\left(  -1/2\right)  }%
=\sum_{k=0}^{n-1}\dbinom{n}{k+1}E_{k}\left(  1\right)  .
\]
We then use
\begin{equation}
E_{k}\left(  1\right)  =-E_{k}\left(  0\right)  =2\left(  2^{k+1}-1\right)
\frac{B_{k+1}}{k+1} \label{15}%
\end{equation}
(\cite[Section 15.3]{QG}) to obtain
\[
\sum_{j=1}^{n}\left(  -1\right)  ^{j}\left(  \dbinom{n}{j}-1\right)
\frac{B_{j}}{j}2^{j}-2\sum_{j=1}^{n}\dbinom{n}{j}\frac{B_{j}}{j}\left(
2^{j}-1\right)  =2n-H_{n}.
\]
Since $\left(  -1\right)  ^{n}B_{n}=B_{n}$ for $n\geq2,$ we arrive at
\[
\sum_{j=1}^{n}\dbinom{n}{j}\frac{B_{j}}{j}\left(  2^{j-1}-1\right)
+\sum_{j=1}^{n}\frac{B_{j}}{j}2^{j-1}=\frac{1}{2}H_{n}-1.
\]
(\ref{16}) now follows after some elementary manipulations.
\end{proof}

Before presenting the next result, we recall the poly-Bernoulli polynomials.
The poly-Bernoulli polynomials are defined via the generating function
\begin{equation}
\sum_{n=0}^{\infty}\mathbb{B}_{n}^{\left(  p\right)  }\left(  x\right)
\frac{t^{n}}{n!}=\frac{\mathrm{Li}_{p}\left(  1-e^{-t}\right)  }{1-e^{-t}%
}e^{xt},\label{gf-pB}%
\end{equation}
(\cite{Bayad}), where $\mathrm{Li}_{p}\left(  z\right)  =\sum_{n=1}^{\infty
}z^{n}/n^{p},$ is the polylogarithm function. In particular, $\mathbb{B}%
_{n}^{\left(  p\right)  }\left(  0\right)  =\mathbb{B}_{n}^{\left(  p\right)
}$ is the $n$th poly-Bernoulli number and when $p=2,$ $\mathbb{B}_{n}^{\left(
2\right)  }$ is called as di-Bernoulli number (\cite{Kaneko}).

We now state a cumulative sum of the Bernoulli numbers in terms of the
Bernoulli numbers, di-Bernoulli numbers, and the values of the di-Bernoulli
polynomials at $x=1$.

\begin{corollary}
For an integer $n\geq2$, we have
\[
\sum_{j=0}^{n}B_{j}=\mathbb{B}_{n}^{\left(  2\right)  }\left(  1\right)
+B_{n}-\mathbb{B}_{n}^{\left(  2\right)  }-1.
\]

\end{corollary}

\begin{proof}
The derivative of (\ref{9}) with respect to $x$ is
\begin{align}
n\sum_{j=1}^{n-1}\left(  \dbinom{n}{j}-1\right)  \frac{B_{j}}{j}x^{n-j-1}  &
-\sum_{j=1}^{n-1}\left(  \dbinom{n}{j}-1\right)  B_{j}x^{n-j-1}\nonumber\\
&  =\frac{d}{dx}\text{ }_{H}w_{n}^{\left(  x\right)  }-nH_{n}x^{n-1}.
\label{MC1}%
\end{align}
By (\ref{9}), we have
\begin{equation}
n\sum_{j=1}^{n}\left(  \dbinom{n}{j}-1\right)  \frac{B_{j}}{j}x^{n-j-1}%
=\frac{n}{x}\left(  \text{ }_{H}w_{n}^{\left(  x\right)  }-H_{n}x^{n}\right)
. \label{18}%
\end{equation}
The Bernoulli polynomials $B_{n}\left(  x\right)  $ are defined by
\[
\sum_{j=0}^{n}\dbinom{n}{j}B_{j}x^{n-j}=B_{n}\left(  x\right)  .
\]
Thus, we have
\begin{equation}
\sum_{j=1}^{n-1}\left(  \dbinom{n}{j}-1\right)  B_{j}x^{n-j-1}=\frac{1}%
{x}B_{n}\left(  x\right)  -\frac{1}{x}\sum_{j=1}^{n-1}B_{j}x^{n-j}.
\label{MC2}%
\end{equation}
Now, using the obvious relation
\[
\frac{d}{dx}\dbinom{x}{k}=\dbinom{x}{k}\sum_{j=0}^{k-1}\frac{1}{x-j}%
\]
in (\ref{17}), we find that
\begin{equation}
\frac{d}{dx}\text{ }_{H}w_{n}^{\left(  x\right)  }=\sum_{k=1}^{n}%
\genfrac{\{}{\}}{0pt}{0}{n}{k}%
\dbinom{x}{k}k!H_{k}\sum_{j=0}^{k-1}\frac{1}{x-j}. \label{MC3}%
\end{equation}
Inserting (\ref{18}), (\ref{MC2}), and (\ref{MC3}) in (\ref{MC1}) gives
\begin{align*}
\sum_{j=0}^{n}B_{j}x^{n-j}  &  =x\sum_{k=1}^{n}%
\genfrac{\{}{\}}{0pt}{0}{n}{k}%
\dbinom{x}{k}k!H_{k}\sum_{j=0}^{k-1}\frac{1}{x-j}-nH_{n}x^{n}\\
&  -n\sum_{j=1}^{n-1}\left(  \dbinom{n}{j}-1\right)  \frac{B_{j}}{j}%
x^{n-j}+B_{n}\left(  x\right)  .
\end{align*}
Setting $x=-1$ and using the relation $\left(  -1\right)  ^{n}B_{n}\left(
-x\right)  =B_{n}\left(  x\right)  +nx^{n-1}$ in the above equation yields to
\begin{equation}
\sum_{j=0}^{n}B_{j}=\sum_{k=1}^{n}\left(  -1\right)  ^{n-k}%
\genfrac{\{}{\}}{0pt}{0}{n}{k}%
k!\left(  H_{k}\right)  ^{2}+B_{n}\left(  1\right)  +n-n^{2}-1. \label{14}%
\end{equation}
To evaluate the first sum on the right-hand side of (\ref{14}), we use
(\ref{2sgf}) and find that%
\begin{equation}
\sum_{n=1}^{\infty}\frac{t^{n}}{n!}\sum_{k=1}^{n}\left(  -1\right)  ^{n-k}%
\genfrac{\{}{\}}{0pt}{0}{n}{k}%
k!\left(  H_{k}\right)  ^{2}=\sum_{k=1}^{\infty}\left(  H_{k}\right)
^{2}\left(  1-e^{-t}\right)  ^{k}. \label{13}%
\end{equation}
With the use of the generating function of the harmonic numbers
\[
\sum_{k=1}^{\infty}H_{k}t^{k}=-\frac{\ln\left(  1-x\right)  }{1-x},
\]
one can obtain that
\[
\sum_{k=1}^{\infty}\left(  H_{k}\right)  ^{2}t^{k}=\frac{\mathrm{Li}%
_{2}\left(  x\right)  }{1-x}+\frac{\ln^{2}\left(  1-x\right)  }{1-x}.
\]
Using this relation with (\ref{gf-pB}) in (\ref{13}), we arrive at%
\[
\sum_{k=1}^{n}\left(  -1\right)  ^{n-k}%
\genfrac{\{}{\}}{0pt}{0}{n}{k}%
k!\left(  H_{k}\right)  ^{2}=\mathbb{B}_{n}^{\left(  2\right)  }\left(
1\right)  -\mathbb{B}_{n}^{\left(  2\right)  }+n\left(  n-1\right)  .
\]
Combining these results in (\ref{14}) gives the desired formula.
\end{proof}

We conclude this section by presenting some congruences for cumulative sums of
the Bernoulli and Euler numbers. Here we call $E_{n}\left(  0\right)  =E_{n}$
is the $n$th Euler number.

\begin{corollary}
For an odd prime $p$, we have%

\begin{align}
\sum_{j=0}^{p}pB_{j}  &  \equiv-1\text{ }\left(  \operatorname{mod}p\right)
,\label{c1}\\
\sum_{j=0}^{p-3}B_{j}  &  \equiv-1\text{ }\left(  \operatorname{mod}p\right)
,\text{ \ \ }\left(  p\geq5\right)  ,\label{c4}\\
\sum_{j=0}^{p}E_{j}  &  \equiv\frac{3}{2}\text{ }\left(  \operatorname{mod}%
p\right)  ,\label{c2}\\
p\sum_{j=0}^{p}\frac{B_{j}}{\left(  p-j+1\right)  }  &  \equiv-1\text{
}\left(  \operatorname{mod}p\right)  . \label{c3}%
\end{align}

\end{corollary}

\begin{proof}
Let $n=p$ be an odd prime in (\ref{14}). Since $p\geq3$, we have $B_{p}\left(
1\right)  =B_{p}=0.$ Thus
\begin{align*}
\sum_{j=0}^{p}B_{j}  &  =\sum_{k=1}^{p}\left(  -1\right)  ^{p-k}%
\genfrac{\{}{\}}{0pt}{0}{p}{k}%
k!\left(  H_{k}\right)  ^{2}+B_{p}\left(  1\right)  +p-p^{2}-1\\
&  =p-p^{2}-1+%
\genfrac{\{}{\}}{0pt}{0}{p}{1}%
\left(  H_{1}\right)  ^{2}+%
\genfrac{\{}{\}}{0pt}{0}{p}{p}%
p!\left(  H_{p}\right)  ^{2}+\sum_{k=2}^{p-1}\left(  -1\right)  ^{p-k}%
\genfrac{\{}{\}}{0pt}{0}{p}{k}%
k!\left(  H_{k}\right)  ^{2}.
\end{align*}
Now, $%
\genfrac{\{}{\}}{0pt}{0}{p}{1}%
=%
\genfrac{\{}{\}}{0pt}{0}{p}{p}%
=1$, $H_{1}=1$, and
\[
p!\left(  H_{p}\right)  ^{2}=p!\left(  H_{p-1}\right)  ^{2}+2\left(
p-1\right)  !H_{p-1}+\frac{\left(  p-1\right)  !}{p},
\]
so
\[
\sum_{j=0}^{p}B_{j}=p-p^{2}+p!\left(  H_{p-1}\right)  ^{2}+2\left(
p-1\right)  !H_{p-1}+\frac{\left(  p-1\right)  !}{p}+\sum_{k=2}^{p-1}\left(
-1\right)  ^{p-k}%
\genfrac{\{}{\}}{0pt}{0}{p}{k}%
k!\left(  H_{k}\right)  ^{2},
\]
or equivalently%
\begin{align}
\sum_{j=0}^{p}pB_{j}  &  =p^{2}-p^{3}+p\left(  p!\right)  \left(
H_{p-1}\right)  ^{2}+2p! H_{p-1}+\left(  p-1\right)  !\nonumber\\
&  +p\sum_{k=2}^{p-1}\left(  -1\right)  ^{p-k}%
\genfrac{\{}{\}}{0pt}{0}{p}{k}%
k!\left(  H_{k}\right)  ^{2}. \label{MC4}%
\end{align}
Since $%
\genfrac{\{}{\}}{0pt}{0}{p}{k}%
\equiv0$ $\left(  \operatorname{mod}p\right)  $ for $k=2,3,\cdots,p-1$, we
find that
\[
\sum_{j=0}^{p}pB_{j}\equiv\left(  p-1\right)  !\equiv-1\text{ }\left(
\operatorname{mod}p\right)  ,
\]
which is the congruence given in (\ref{c1}).

Now, Babbage's theorem states that $H_{p-1}\equiv0$ $\left(
\operatorname{mod}p\right)  $, so (\ref{MC4}) can be written as
\[
\sum_{j=0}^{p}pB_{j}\equiv\left(  p-1\right)  !\text{ }\left(
\operatorname{mod}p^{2}\right)  .
\]
A famous congruence by Glaisher \cite{Glaisher1899} is
\[
\left(  p-1\right)  !\equiv-p+pB_{p-1}\text{ }\left(  \operatorname{mod}%
p^{2}\right)  .
\]
Thus%
\begin{align*}
-p+pB_{p-1}  &  \equiv\sum_{j=0}^{p}pB_{j}=pB_{0}+pB_{1}+pB_{p}+pB_{p-1}%
+pB_{p-2}+\sum_{j=2}^{p-3}pB_{j}\\
&  =p-\frac{p}{2}+pB_{p-1}+\sum_{j=2}^{p-3}pB_{j}\left(  \operatorname{mod}%
p^{2}\right)
\end{align*}
which implies that%
\[
\sum_{j=2}^{p-3}pB_{j}\equiv-\frac{3p}{2}\text{ }\left(  \operatorname{mod}%
p^{2}\right)
\]
for $p\geq3$, or equivalently%
\[
\sum_{j=2}^{p-3}B_{j}\equiv-\frac{3}{2}\text{ }\left(  \operatorname{mod}%
p\right)
\]
if $p\geq5.$ (\ref{c4}) now follows by using the values $B_{0}$ and $B_{1}$.

Inserting (\ref{15}) into (\ref{16}) yields to%
\[
\sum_{j=1}^{n}\left(  \binom{n}{j}+1\right)  \frac{E_{j-1}}{2}=1,
\]
or%
\[
\sum_{j=0}^{n-1}E_{j}=2-nE_{n}-E_{n-1}-\sum_{j=2}^{n-1}\binom{n}{j}E_{j-1}.
\]
Let $n=p$ be an odd prime. Then, we have%
\[
\sum_{j=0}^{p}E_{j}=2+E_{p}-p-\sum_{j=2}^{p-1}\binom{p}{j}E_{j-1},
\]
since $E_{0}=1$ and $E_{2m}=0$ for $m\geq1.$ Moreover, $2^{m}E_{m}\in%
\mathbb{Z}
,$ which means that $E_{m}$ does not contain primes $p$ in the denominator
when $p>2.$ Thus
\[
\binom{p}{j}E_{j-1}\equiv0\text{ }\left(  \operatorname{mod}p\right)
\]
for $j=1,2,\cdots,p-1.$ Therefore, we have
\begin{equation}
\sum_{j=0}^{p}E_{j}\equiv2-\frac{1}{2}=\frac{3}{2}\text{ }\left(
\operatorname{mod}p\right)  \label{22}%
\end{equation}
for $p\geq5.$ Note that
\[
\sum_{j=0}^{3}E_{j}=E_{0}+E_{1}+E_{2}+E_{3}=1-\frac{1}{2}+\frac{1}{4}=\frac
{3}{4},
\]
hence, (\ref{22}) is also true when $p=3$, that is, we have (\ref{c2}).

To prove (\ref{c3}), we first integrate both sides of (\ref{9}) with respect
to $x$ from $0$ to $1$ and use (\ref{12}) and the well-known formula%
\[
\sum_{j=0}^{n}\dbinom{n}{j}\frac{B_{j}}{n-j+1}=\int\limits_{0}^{1}B_{n}\left(
x\right)  dx=0,\text{ }n\geq2.
\]
We then reach at
\[
\sum_{j=0}^{n}\frac{B_{j}}{\left(  n-j+1\right)  }=1-\left(  n+1\right)
\int\limits_{0}^{1}\text{ }_{H}w_{n}^{\left(  x\right)  }dx.
\]
From the definition of the Cauchy numbers of the first kind $c_{n}$ (cf.
\cite{MSV})%
\[
\frac{c_{n}}{n!}=\int\limits_{0}^{1}\dbinom{x}{n}dx
\]
we arrive at%
\[
\sum_{j=0}^{n}\frac{B_{j}}{\left(  n-j+1\right)  }=1-\left(  n+1\right)
\sum_{k=1}^{n}%
\genfrac{\{}{\}}{0pt}{0}{n}{k}%
c_{k}H_{k}.
\]
We write $n=p,$ where $p$ is an odd prime in the equation above to obtain%
\begin{equation}
\sum_{j=0}^{p}\frac{B_{j}}{\left(  p-j+1\right)  }=1-\left(  p+1\right)
\sum_{k=1}^{p}%
\genfrac{\{}{\}}{0pt}{0}{p}{k}%
c_{k}H_{k}. \label{19}%
\end{equation}
Employing $%
\genfrac{\{}{\}}{0pt}{0}{p}{1}%
=%
\genfrac{\{}{\}}{0pt}{0}{p}{p}%
=1$ and $%
\genfrac{\{}{\}}{0pt}{0}{p}{p-1}%
=\dbinom{p}{2}$ and the explicit expression of Cauchy numbers of the first
kind (\cite{MSV})
\begin{equation}
c_{k}=\sum_{j=1}^{k}%
\genfrac{[}{]}{0pt}{0}{k}{j}%
\frac{\left(  -1\right)  ^{k-j}}{j+1}, \label{MC5}%
\end{equation}
we have%
\begin{align*}
\sum_{k=1}^{p}%
\genfrac{\{}{\}}{0pt}{0}{p}{k}%
c_{k}H_{k}  &  =\frac{1}{2}+c_{p}H_{p-1}+\frac{c_{p}}{p}+\frac{p\left(
p-1\right)  }{2}c_{p-1}H_{p-1}\\
&  +\sum_{k=2}^{p-2}\sum_{j=1}^{k}%
\genfrac{\{}{\}}{0pt}{0}{p}{k}%
\genfrac{[}{]}{0pt}{0}{k}{j}%
\frac{\left(  -1\right)  ^{k-j}}{j+1}H_{k},
\end{align*}
or equivalently%
\begin{align}
p\sum_{k=1}^{p}%
\genfrac{\{}{\}}{0pt}{0}{p}{k}%
c_{k}H_{k}  &  =\frac{p}{2}+pc_{p}H_{p-1}+c_{p}+\frac{p\left(  p-1\right)
}{2}pc_{p-1}H_{p-1}\nonumber\\
&  +p\sum_{k=2}^{p-2}\sum_{j=1}^{k}%
\genfrac{\{}{\}}{0pt}{0}{p}{k}%
\genfrac{[}{]}{0pt}{0}{k}{j}%
\frac{\left(  -1\right)  ^{k-j}}{j+1}H_{k.} \label{21}%
\end{align}
Let%
\[
A\left(  p\right)  =\sum_{k=2}^{p-2}\sum_{j=1}^{k}%
\genfrac{\{}{\}}{0pt}{0}{p}{k}%
\genfrac{[}{]}{0pt}{0}{k}{j}%
\frac{\left(  -1\right)  ^{k-j}}{j+1}H_{k.}.
\]
The product $\frac{1}{j+1}H_{k}$ does not involve $p$ in the denominator, so
$A\left(  p\right)  \equiv0\left(  \operatorname{mod}p\right)  $ since $%
\genfrac{\{}{\}}{0pt}{0}{p}{k}%
\equiv0\left(  \operatorname{mod}p\right)  $ for $k=2,3,\cdots,p-2.$
Therefore, from (\ref{19}) and (\ref{21}), we conclude that%
\begin{align}
p\sum_{j=0}^{p}\frac{B_{j}}{\left(  p-j+1\right)  }  &  =p-\frac{p^{2}}%
{2}-\frac{p}{2}-p^{2}c_{p}H_{p-1}-pc_{p}H_{p-1}-pc_{p}-c_{p}\nonumber\\
&  -\frac{p^{2}\left(  p-1\right)  }{2}pc_{p-1}H_{p-1}-\frac{p\left(
p-1\right)  }{2}pc_{p-1}H_{p-1}-p^{2}A\left(  p\right)  -pA\left(  p\right)  .
\label{23}%
\end{align}
Now, in Corollary 5.2 of \cite{CK}, it has been recorded that if
$\lambda\equiv0$ $\left(  \operatorname{mod}p\right)  ,$ then
\[
c_{p}\left(  \lambda\right)  \equiv1-\lambda+\left(  1-\lambda\right)  \left(
-\lambda\right)  \cdots\left(  2-\lambda-p\right)  \text{ }\left(
\operatorname{mod}p\right)
\]
and%
\[
pc_{p-1}\left(  \lambda\right)  \equiv1\text{ }\left(  \operatorname{mod}%
p\right)  ,
\]
where $c_{k}\left(  \lambda\right)  $ is the degenerate Cauchy number. Since
$\underset{\lambda\rightarrow0}{\lim}c_{k}\left(  \lambda\right)  =c_{k}$ , we
conclude that%
\[
c_{p}\equiv1\text{ }\left(  \operatorname{mod}p\right)  \text{ and }%
pc_{p-1}=\underset{\lambda\rightarrow0}{\lim}pc_{p-1}\left(  \lambda\right)
\equiv1\text{ }\left(  \operatorname{mod}p\right)
\]
Therefore, (\ref{23}) implies the congruence (\ref{c3}).
\end{proof}

\begin{remark}
(\ref{c1}) also follows from the von Staudt-Clausen theorem for the Bernoulli
numbers:
\[
pB_{2j}\equiv\left\{
\begin{array}
[c]{rl}%
0\left(  \operatorname{mod}p\right)  , & \text{if }\left(  p-1\right)
\nmid2j,\\
-1\left(  \operatorname{mod}p\right)  , & \text{if }\left(  p-1\right)  |2j.
\end{array}
\right.
\]
Indeed, we first write
\begin{align*}
\sum_{j=0}^{p}pB_{j}  &  =pB_{0}+pB_{1}+pB_{p}+pB_{p-1}+pB_{p-2}+\sum
_{j=1}^{\frac{p-3}{2}}pB_{2j}\\
&  =p-\frac{p}{2}+pB_{p-1}+\sum_{j=1}^{\frac{p-3}{2}}pB_{2j}.
\end{align*}
As j ranges from $1$ to $\frac{p-3}{2},$ $\left(  p-1\right)  $ does not
divide $2j.$ So $pB_{2j}\equiv0\left(  \operatorname{mod}p\right)  $ for
$j=1,2,\cdots,\frac{p-3}{2}.$ Moreover, $pB_{p-1}\equiv-1\left(
\operatorname{mod}p\right)  $ and hence we have
\[
\sum_{j=0}^{p}pB_{j}\equiv-1\text{ }\left(  \operatorname{mod}p\right)
\]
for an odd prime $p.$

We also note that, from (\ref{23}), the congruence (\ref{c3}) is equivalent
to
\[
p\sum_{j=0}^{p}\frac{B_{j}}{\left(  p-j+1\right)  }=-\frac{p}{2}-c_{p}\text{
}\left(  \operatorname{mod}p^{2}\right)  .
\]

\end{remark}

\section{Proof of the Main Theorem}

\setcounter{theorem}{0} \setcounter{equation}{0}

Let
\[
\sum_{k=j}^{n}\left(  -1\right)  ^{k-j}%
\genfrac{\{}{\}}{0pt}{0}{n}{k}%
\genfrac{[}{]}{0pt}{0}{k}{j}%
H_{k}=\mathcal{B}_{n,j}.
\]
Using the well known recurrence relation for the Stirling numbers of the
second kind \cite[Eq. (9.1)]{QG}
\[%
\genfrac{\{}{\}}{0pt}{0}{n+1}{k}%
=%
\genfrac{\{}{\}}{0pt}{0}{n}{k}%
k+%
\genfrac{\{}{\}}{0pt}{0}{n}{k-1}%
,
\]
we obtain that
\begin{align}
\mathcal{B}_{n+1,j}  &  =\sum_{k=j}^{n+1}\left(  -1\right)  ^{k-j}%
\genfrac{\{}{\}}{0pt}{0}{n}{k}%
\genfrac{[}{]}{0pt}{0}{k}{j}%
kH_{k}\nonumber\\
&  +\sum_{k=j}^{n+1}\left(  -1\right)  ^{k-j}%
\genfrac{\{}{\}}{0pt}{0}{n}{k-1}%
\genfrac{[}{]}{0pt}{0}{k}{j}%
H_{k}. \label{3}%
\end{align}
Applying the well-known recurrence relation of the Stirling numbers of the
first kind \cite[Eq. (12.3)]{QG}
\[%
\genfrac{[}{]}{0pt}{0}{k+1}{j}%
=%
\genfrac{[}{]}{0pt}{0}{k}{j}%
k+%
\genfrac{[}{]}{0pt}{0}{k}{j-1}%
,
\]
and the fact that $%
\genfrac{\{}{\}}{0pt}{0}{n}{k}%
=0$ for $k>n,$ (\ref{3}) becomes
\begin{align*}
\mathcal{B}_{n+1,j}  &  =\sum_{k=j}^{n}\left(  -1\right)  ^{k-j}%
\genfrac{\{}{\}}{0pt}{0}{n}{k}%
\genfrac{[}{]}{0pt}{0}{k+1}{j}%
H_{k}-\sum_{k=j}^{n}\left(  -1\right)  ^{k-j}%
\genfrac{\{}{\}}{0pt}{0}{n}{k}%
\genfrac{[}{]}{0pt}{0}{k}{j-1}%
H_{k}\\
&  -\sum_{k=j-1}^{n}\left(  -1\right)  ^{k-j}%
\genfrac{\{}{\}}{0pt}{0}{n}{k}%
\genfrac{[}{]}{0pt}{0}{k+1}{j}%
H_{k}-\sum_{k=j-1}^{n}\left(  -1\right)  ^{k-j}%
\genfrac{\{}{\}}{0pt}{0}{n}{k}%
\genfrac{[}{]}{0pt}{0}{k+1}{j}%
\frac{1}{k+1}.
\end{align*}
Employing the identity (\ref{4}) in the form
\[
\sum_{k=j-1}^{n-1}\frac{\left(  -1\right)  ^{k+1-j}}{k+1}%
\genfrac{\{}{\}}{0pt}{0}{n-1}{k}%
\genfrac{[}{]}{0pt}{0}{k+1}{j}%
=\dbinom{n-1}{j}\frac{B_{n-j}}{n-j}\,\ \left(  1\leq j\leq n\right)  ,
\]
we find that
\[
\mathcal{B}_{n+1,j}=\mathcal{B}_{n,j-1}+\dbinom{n}{j}\frac{B_{n+1-j}}{n+1-j}.
\]
From this reduction formula, we see that
\[
\mathcal{B}_{n+1,j}=\mathcal{B}_{n-1,j-2}+\left(  \dbinom{n}{j}+\dbinom
{n-1}{j-1}\right)  \frac{B_{n+1-j}}{n+1-j}.
\]
Consecutive $j-2$ times applications yield to
\begin{equation}
\mathcal{B}_{n+1,j}=\mathcal{B}_{n-j,0}+\frac{B_{n+1-j}}{n+1-j}\sum
_{k=0}^{j-1}\dbinom{n-k}{j-k}. \label{2}%
\end{equation}
Here, we have
\begin{equation}
\sum_{k=0}^{j-1}\dbinom{n-k}{j-k}=\dbinom{n+1}{j}-1, \label{8}%
\end{equation}
upon the use of \cite[p. 7, Eq. (1.49)]{G}%
\[
\sum_{k=0}^{n}\dbinom{x+k}{k}=\dbinom{x+n+1}{n}.
\]
Substituting (\ref{8}) into (\ref{2}), with the use of $\mathcal{B}_{n-j,0}%
=0$, we arrive at%
\[
\mathcal{B}_{n+1,j}=\left(  \dbinom{n+1}{j}-1\right)  \frac{B_{n+1-j}}%
{n+1-j},
\]
which is the desired result.

\end{document}